\documentclass[12pt,a4paper]{amsart}
\pdfoutput=1

\usepackage{amsmath,amssymb,amsthm} 

\usepackage{mathtools}
\usepackage{tikz}
\usepackage{blindtext}
\usepackage{xcolor} 	
\usepackage{hyperref}
\hypersetup{
	colorlinks,
	linkcolor={red!60!black},
	citecolor={green!60!black},
	urlcolor={blue!60!black},
}
\usepackage[abbrev, msc-links]{amsrefs}

\usepackage{soul}

\usepackage[utf8]{inputenc}
\usepackage[T1]{fontenc}
\usepackage{lmodern}
\usepackage[babel]{microtype}
\usepackage[english]{babel}

\usepackage{geometry}
\theoremstyle{plain}
\newtheorem{thm}{Theorem}[section]

\newtheorem{prop}[thm]{Proposition}

\newtheorem{cor}[thm]{Corollary}

\newtheorem{obs}[thm]{Observation}

\newtheorem{quest}[thm]{Question}

\newtheorem*{thm*}{Theorem}
\newtheorem*{cor*}{Corollary}

\theoremstyle{definition}
\newtheorem{rem}[thm]{Remark}

\usepackage{enumitem}

\author{Nathan Bowler}
\address{Nathan Bowler, Universität Hamburg, Department of Mathematics, Bundesstra{\ss}e 55 (Geomatikum), 20146 
	Hamburg, Germany}
\email{nathan.bowler@uni-hamburg.de}

\author{Florian Gut}
\address{Florian Gut, Universität Hamburg, Department of Mathematics, Bundesstra{\ss}e 55 (Geomatikum), 20146 Hamburg, 
	Germany}
\email{florian.gut@uni-hamburg.de}

\author{Attila Jo{\'o}}
\thanks{The third author would like to thank the generous support of the Alexander 
	von Humboldt Foundation and NKFIH 
	OTKA-129211}
\address{Attila Jo{\'o}, Universität Hamburg, Department of Mathematics, Bundesstra{\ss}e 55 (Geomatikum), 20146 
	Hamburg, Germany and  Alfr\'{e}d R\'{e}nyi Institute of Mathematics, Logic, Set theory and topology department,  13-15 
	Re\'{a}ltanoda St., 
	Budapest, Hungary}
\email{attila.joo@uni-hamburg.de\\ jooattila@renyi.hu}

\author{Max Pitz}
\address{Max Pitz, Universität Hamburg, Department of Mathematics, Bundesstra{\ss}e 55 (Geomatikum), 20146 Hamburg, 
	Germany}
\email{max.pitz@uni-hamburg.de}

\begin{document}
	
	\title{Maker-Breaker games on \texorpdfstring{$ K_{\omega_1}$}{Kw1}  and \texorpdfstring{$K_{\omega,\omega_1}$}{Kww1} }
	\begin{abstract}
		We investigate Maker-Breaker games on graphs of size $\aleph_1$ in which Maker's goal is to build a copy of the host 
		graph. We establish a firm dependence of the outcome of the game on the axiomatic framework.
		Relating to this, we prove that there is a winning strategy for Maker in the $K_{\omega,\omega_1}$-game under 
		ZFC+MA+$\neg$CH  and a winning strategy for Breaker under ZFC+CH.
		We prove a similar result for the $K_{\omega_1}$-game.
		Here, Maker has a winning strategy under ZF+DC+AD, while Breaker has one under ZFC+CH again.
	\end{abstract}
	\maketitle
	\section{Introduction}
	Games on graphs are a very natural concept and so it is no wonder that this field has emerged jointly with graph theory as a whole.
	For finite boards one often considers strong games, i.e. where two players interchangeably colour edges of a finite graph $G$ with the aim to be the first player to have some previously agreed upon graph contained as a subgraph in the graph induced by their respective coloured edges.
	Another important kind of games is the so-called ``Maker-Breaker games''.
	A typical setup for such games on  (potentially infinite) graphs is the following: at each turn, Maker claims an edge of $ G $ (not previously claimed by either player) after which Breaker claims an unclaimed edge.
	There is either a fixed number of turns or they play until the whole edge set is distributed. The set of winning sets of Maker is public knowledge and usually has some combinatorial description. 
	Classical games of this type are for example the Shannon switching game, in which Maker's goal is to connect two vertices by a path (see \cite{Lehman_1964}), and the game where Maker's goal is to build a spanning tree (see \cite{chvatal1978biased}) or more generally a base of a matroid (see \cite{lgorzata2005biased}). 
	
	For recent results about Maker-Breaker games on infinite graphs  we refer to \cite{NICHOLASDAY2021482} and \cite{bowler2020maker}.
	Some games (like the base-exchange game in \cite{aharoni1991bases}) can be phrased more naturally in the language of infinite matroids.
	It is worth mentioning that Maker-Breaker games have been investigated in an even more abstract settings as well (see \cite{CN1979}).
	
	For graphs $G$ and $H$, let $\mathfrak{MB}(G, H) $ denote the \emph{Maker-Breaker game} where $G$ (more precisely the set of edges) is the board, there are turns (indexed by ordinals) each of which begins with Maker claiming a previously unclaimed edge, after which Breaker does likewise.
	The game terminates when all the edges are claimed and Maker wins if and only if at the end of the game the subgraph $ G_M $ of $ G $ induced by the edges claimed by Maker contains a subgraph isomorphic to $ H $.
	Let us recall that a graph $ G $ is an ordered pair $ (V,E) $ with $ E\subseteq [V]^{2} $ where $ V $ is called the vertex set and $ E $ is the edge set of $ G $.
	The complete graph on $ \kappa $ is $ K_\kappa:=(\kappa, [\kappa]^{2}) $.
	The complete bipartite graph with vertex classes of size $ \lambda $ and $ \kappa $ is denoted by  $ K_{\lambda,\kappa} $, its vertex set is $ (\lambda \times \{ 0 \})\cup (\kappa\times \{ 1 \}) $ and its edge set is $ \{ (\alpha,0), (\beta, 1):\ \alpha<\lambda, \beta<\kappa \} $.
	Note that the vertex sets of these graphs are already well-ordered, and so we generally do not need to invoke the axiom of choice.
	It was shown that in the game $ \mathfrak{MB}(K_{\omega}, K_{\omega}) $ Maker has a winning strategy (see \cite{bowler2020maker}).
	In this note we analyse similar games on uncountable graphs.
	
	Note that each outcome of the game defines a 2-colouring  of $E(G)$.
	This suggests a possible connection to Ramsey type problems, although in the current context the colourings in question are not arbitrary but are produced by players with particular goals in mind.
	There are colourings of the edges of a $K_{\omega_1}$ with two colours without any monochromatic $K_{\omega_1}$ in ZFC (see~\cite{sierpinski1933}), but if instead of the axiom of choice one assumes DC+AD, then there is always a monochromatic $K_{\omega_1}$ because $\omega_1$ becomes measurable (see \cite{kanamori2008higher}*{Theorem 28.2})  and hence weakly compact\footnote{We write CH, GCH, DC, AD and $ \mathfrak{p} $ for the continuum hypothesis, generalised continuum hypothesis, axiom of dependent choice, axiom of determinacy and the pseudo-intersection number respectively.}. 
	
	The existence of a monochromatic $K_{\omega,\omega_1}$ when colouring the edges of a $K_{\omega,\omega_1}$ with two colours is even more dependent on the set-theoretic framework.
	While there is a colouring without a monochromatic copy in ZFC+CH, there is no such colouring in ZFC+$\omega_1<\mathfrak{p} $.
	Since we could not find these particular statements formulated anywhere in the literature on infinite Ramsey theory, for the sake of completeness we include them here as Corollaries \ref{cor: no monochrom} and \ref{cor: monochrom}.
	
	These Ramsey-type results compare well to the corresponding results about the existence of a winning strategy for either player.
	Our main results are as follows:
	
	\begin{thm}\label{t: main thm bip}
		It is independent of ZFC if Breaker has a winning strategy in the game $\mathfrak{MB}(K_{\omega, \omega_1}, K_{\omega,\omega_1}) $. He has one under ZFC+GCH\footnote{A closer analysis shows that only CH is needed here, but we have chosen a simpler exposition over optimality of the results, since the independence is our main concern.},  while Maker has one under ZFC$+\omega_1<\mathfrak{p}$.
	\end{thm}
	
	\begin{thm}\label{t: Ramsey bipart}
		It is independent of ZFC if every  $ 2 $-colouring of the edges of $ K_{\omega, \omega_1} $ admits a monochromatic copy of $ K_{\omega, \omega_1} $. It is true in ZFC$+\omega_1<\mathfrak{p} $ but fails under ZFC+CH.
	\end{thm}
	
	\begin{thm}\label{t: main thm}
		Assuming the consistency of AD, it is independent of ZF+DC  if Breaker has winning strategies in the games $\mathfrak{MB}(K_{\omega_n}, K_{\omega_n}) $ for $ n\in \{ 1,2 \} $. He has such winning strategies under ZFC+GCH,  while Maker has winning strategies in these games under ZF+DC+AD.
	\end{thm}

	Let $\mathfrak{MB}(K_{\kappa}, K_{\mathsf{club}}) $ be the game in which Maker's goal is to build a ``$K_{\mathsf{club}}$'', i.e. a  complete graph whose vertex set is a closed unbounded subset of $ \kappa $.
	
	\begin{thm}\label{t: main thm club}
		Assuming the consistency of AD, it is independent of ZF+DC  if Breaker has a winning strategy in the game $\mathfrak{MB}(K_{\omega_1}, K_{\mathsf{club}}) $.
	\end{thm}
	
	Our results raise the following natural questions: 
	
	\begin{quest}
		Is it consistent with ZFC that neither Maker nor Breaker has a winning strategy in the game $\mathfrak{MB}(K_{\omega, \omega_1}, K_{\omega,\omega_1}) $?
	\end{quest}
	
	\begin{quest}
		Does Breaker have a winning strategy in $\mathfrak{MB}(K_{\omega_1}, K_{\omega_1})$ under ZFC?
	\end{quest}
	
	\begin{quest}
		Does Maker have a winning strategy in $\mathfrak{MB}(K_{\omega_1}, K_{\mathsf{club}}) $ under  ZF+DC+AD?
	\end{quest}
	
	\textbf{Acknowledgements:} The authors are grateful to Stefan Geschke, Zoltán Vidnyánszky and Daniel Hathaway for the insightful discussions about the Axiom of determinacy.

	\section{The winning strategies of Breaker under GCH}
	\begin{prop}[ZFC+GCH]\label{p: Breaker win}
		For every infinite cardinal $ \kappa $, Breaker has a winning strategy in the game $\mathfrak{MB}(K_{\kappa^{+}}, K_{\kappa, \kappa^{+}}) $.
	\end{prop}
	\begin{proof}
		Let us assume that $ K_{\kappa^{+}} $ is represented as the complete graph on the vertex set $\kappa^+$. Working under GCH, we fix an  enumeration $ \{ A_\alpha:
		\alpha<\kappa^{+}   \} $ of $ [\kappa^{+}]^{\kappa} $  and for each  $\alpha<\kappa^{+}$, we pick a surjective function $ f_\alpha: \kappa \rightarrow \{ A_\beta: \beta \leq\alpha \}$).
		Whenever Maker plays an edge $ \{ \beta, \alpha \} $ with $ \beta < \alpha $  and there is a $ \gamma<\kappa $ such that this is the $(\gamma+1) $st downwards edge from $ \alpha $ she claims, Breaker chooses  the smallest $\delta\in f_\alpha(\gamma) $ for which $ \{ \delta, \alpha \} $ is available, and plays $ \{ \delta, \alpha \} $ if such a $ \delta $ exists - otherwise he plays arbitrarily. 
		
		Suppose for a contradiction that Maker manages to build a $ K_{\kappa, \kappa^{+}} $ (despite Breaker playing as above) and let $ A $ be its smaller and $ B $ its bigger vertex class.
		Then there is an  $ \alpha<\kappa^{+} $ with $A_{\alpha}=A$. 
		Fix a $ \beta\in  B $ with $ \beta > \max \{\alpha, \sup A\} $ and let $ \gamma < \kappa $ with $ f_\beta(\gamma)=A $.
		At the turn when Maker claims a downwards edge from $ \beta $ for the $ (\gamma+1) $st time, there are still $ \kappa $ many  $ \delta\in A $ for which $ \{ \delta, \beta \} $ is available, thus Breaker's next play is $ \{ \delta, \beta \} $ for the smallest such $ \delta $.
		This contradicts  $ \{ \delta, \beta \}\in E(G_M) $.
	\end{proof}
	
	The corresponding negative Ramsey-result can be proved in a similar manner:
	\begin{cor}[ZFC+GCH]\label{cor: no monochrom}
		For every infinite cardinal $ \kappa $, there exists a $ 2 $-colouring of the edge set of $ K_{\kappa, \kappa^{+}} $ without a monochromatic copy of $ K_{\kappa, \kappa^{+}} $.
	\end{cor}
	
	\begin{proof}
		Let $ \{ v_\alpha: \alpha<\kappa^{+} \} $ be an enumeration of the larger vertex class and let $ \{ A_\alpha: 
		\alpha<\kappa^{+}   \} $ be an enumeration of $ [\kappa^{+}]^{\kappa} $.
		For each $ \alpha<\kappa^{+} $, we colour the edges incident with $ v_\alpha $ in such a way that for every $ \beta \leq \alpha $ both colours appear among the edges between $ v_\alpha $ and $ A_\beta $.
		This clearly ensures that no  set $ A  $ can be the smaller vertex class of a monochromatic copy of $ K_{\kappa, \kappa^{+}} $ and therefore no such a monochromatic copy exists.
	\end{proof}
	
	\begin{obs}\label{o: monoton}
		If Breaker has a winning strategy in $ \mathfrak{MB}(G,H) $, then he also has one in every game $ \mathfrak{MB}(G',H') $ where $ G' $ is a subgraph of $ G $ and $ H' $ is a supergraph of $ H $.  
	\end{obs}
	
	Since  $ K_{\kappa, \kappa^{+}} $ is a subgraph of $ K_{\kappa^{+}} $, Observation \ref{o: monoton} guarantees that Proposition \ref{p: Breaker win} has the following consequences: 
	\begin{cor}[ZFC+GCH]\label{cor: Breaker win}
		For every infinite cardinal $ \kappa $, Breaker has a winning strategy in the following games:
		\begin{enumerate}
			\item\label{i: Breaker win 1} $\mathfrak{MB}(K_{\kappa, \kappa^{+}}, K_{\kappa, \kappa^{+}}) $,
			\item\label{i: Breaker win 2}  $\mathfrak{MB}(K_{\kappa^{+}}, K_{\kappa^{+}}) $
			\item\label{i: Breaker win 3} $\mathfrak{MB}(K_{\kappa^{+}}, K_{\mathsf{club}}). $
		\end{enumerate}
	\end{cor}

	\section{Winning strategies for Maker}
	
	During the course of play in $\mathfrak{MB}(G,H)$ we will refer to a vertex as {\em fresh} if no edge incident with that vertex has been claimed yet by either player.
	
	\subsection{A winning strategy for Maker in \texorpdfstring{$\mathfrak{MB}(K_{\omega, \omega_1}, K_{\omega, \omega_1})$}{MBKww1}}
	
	A set $\mathcal{F}$ of sets has the strong finite intersection property if the intersection of any finitely many elements of 
	$\mathcal{F}$ is infinite.
	Given two sets $X$ and $Y$, write $X \subseteq^*Y$ if $X \setminus Y$ is finite.
	A \emph{pseudo-intersection} for a set $\mathcal{F}$ of sets is a set $P$ with $P \subseteq^* F$ for all $F \in \mathcal{F}$.
	The cardinal $\mathfrak{p}$ is the minimum cardinality of a set $\mathcal{F}$ of subsets of $\omega$  that has the strong finite intersection property but does not admit  an infinite pseudo-intersection.
	Clearly $\aleph_0<\mathfrak{p} \leq 2^{\aleph_0}$ and it is known that  $ \omega_1<\mathfrak{p} $ is consistent relative to ZFC (see \cite{kunen2011set}*{Lemma III.3.22 on p. 176}). 
	
	\begin{prop}\label{prop: Maker win MA}
		Maker has a winning strategy in $\mathfrak{MB}(K_{\omega, \omega_1}, K_{\omega, \omega_1}) $ if $ \omega_1<\mathfrak{p} $.
	\end{prop}  
	
	\begin{proof}
		Let $ U$ and $ V $ be  the two sides of the bipartite graph $ K_{\omega, \omega_1} $, where $ \left|U\right|=\omega $ and $ \left|V\right|=\omega_1 $.
		We denote the subgraph of $ G $ induced by the edges Maker claimed before turn $ \alpha $  by $ G^{\alpha}_M $ and we write $ N_{G^{\alpha}_{M}}(v) $ for the set of the neighbours of $v$ in this graph.
		
		During the game Maker will choose a sequence $\langle v_{\alpha} \colon \alpha < \kappa \rangle$   $\langle v_{\alpha} \colon \alpha < \omega_1 \rangle$ of distinct vertices from $V$ and a sequence $\langle N_{\alpha} \colon \alpha < \kappa\rangle$ of subsets of $U$ in such a way as to ensure that for any $\alpha < \kappa$ and a sequence $\langle N_{\alpha} \colon \alpha < \omega_1\rangle$ of subsets of $U$ in such a way as to ensure that for any $\alpha < \omega_1$
		\begin{enumerate}
			\item $N_{\alpha} \subseteq N_{G_M}^{\omega \cdot(\alpha + 1)}(v_{\alpha})$.
			\item the set $\{N_{\beta} \colon \beta \leq \alpha \}$ has the strong finite intersection property.
		\end{enumerate}
		
		Assume that  turn $ \alpha \cdot \omega $ has just begun for some $\alpha<\omega_1$ and that Maker has constructed suitable $v_{\beta}$ and $N_{\beta}$ for all $\beta < \alpha$. 
		She picks $ v_{\alpha}$ to be any fresh vertex in $V$. Using (2) for all $\beta < \alpha$, we know that the set $\{N_{\beta} \colon \beta < \alpha \}$ has the strong finite intersection property.
		Let $P_\alpha $ be an infinite pseudo-intersection of this family.
		In each of the next $ \omega $ turns, Maker claims an edge  $ \{ u, v_\alpha \} $ with $ u\in P_\alpha $.
		Let $N_{\alpha}$ be the set of all the endpoints $u \in U$ of these edges.
		It is easy to check that this construction satisfies (1) and (2) for $\alpha$. 
		
		At the end of the game $ \{N_{\alpha} \colon \alpha<\omega_1  \} $ has the strong finite intersection property and hence (by the assumption  $ \omega_1<\mathfrak{p} $) admits an infinite pseudo-intersection $ P $.
		By the definition of $ P $,  for each $ \alpha<\omega_1 $, the set  $ P\setminus N_\alpha $ is finite.
		Then there exists an uncountable $ O\subseteq \omega_1 $ and a finite $ F\subseteq P $  such that $ P\setminus N_\alpha=F $ for every $ \alpha\in O $.
		Finally,$ (P\setminus F)\cup \{ v_\alpha: \alpha\in O \} $ induces a copy of $ K_{\omega, \omega_1} $, all of whose edges have been claimed by Maker.
	\end{proof}
	\begin{rem}
		The same proof shows that Maker has a winning strategy in  $\mathfrak{MB}(K_{\omega, \kappa}, K_{\omega, \kappa}) $ for every $\kappa<\mathfrak{p}$ with $ \mathsf{cf}(\kappa)>\aleph_0 $.
	\end{rem}
	The proof of Proposition \ref{prop: Maker win MA} leads to the following positive Ramsey result:
	\begin{cor}\label{cor: monochrom}
		If $ \omega_1<\mathfrak{p} $, then any $ 2 $-colouring of the edges of $ K_{\omega, \omega_1} $ admits a monochromatic copy of $ K_{\omega, \omega_1} $.
	\end{cor}  
	\begin{proof}
		Call the colours red and blue, and call the countable and uncountable side of the original graph $U$ and $V$ respectively.
		We pick a free ultrafilter $ \mathcal{U} $ on $ U $.
		Then for each $ v\in V $ either the set $ N_r(v) $ of the red neighbours of $ v $ is in $ \mathcal{U} $ or the set $N_b(v)$ of the blue neighbours. We may assume that there is an uncountable $ V'\subseteq V $ such that $ N_r(v)\in \mathcal{U} $ for each $ v\in V' $.
		Since $ \mathcal{U} $ is a free ultrafilter, the family $ \{ N_r(v): v\in V' \} $ has the strong finite intersection property and therefore (by $ \omega_1<\mathfrak{p} $) admits  an infinite pseudo-intersection $ P $. This means that for every $ v\in V' $ the set  $ P\setminus N_r(v) $ is finite.
		Then there exists an uncountable $ V''\subseteq V' $ and finite $ F\subseteq P $ such that $ P\setminus N_r(v)=F $ for each $ v\in V'' $ and hence $ (P\setminus F)\cup V'' $ induces a red copy of $ K_{\omega, \omega_1} $.
	\end{proof}
	\begin{quest}
		Is it consistent with  ZFC$+\aleph_\omega<2^{\aleph_0}$ that Maker has a winning strategy in the game $\mathfrak{MB}(K_{\omega, \omega_\omega}, K_{\omega, \omega_\omega}) $?
	\end{quest}
	
	Theorem \ref{t: main thm bip}  is implied by the case $ \kappa=\omega $  of Corollary  \ref{cor: Breaker win}/ (\ref{i: Breaker win 1})  together with Proposition \ref{prop: Maker win MA}.
	Similarly, Theorem \ref{t: Ramsey bipart} follows from Corollaries \ref{cor: no monochrom} and \ref{cor: monochrom}. 
	
	\subsection{A winning strategy for Maker in \texorpdfstring{$\mathfrak{MB}(K_{\omega_1}, K_{\omega_1})$}{MBKww1} and \texorpdfstring{$\mathfrak{MB}(K_{\omega_2}, K_{\omega_2})$}{MBKww1}}
	
	\begin{prop}[ZF]\label{p: Maker wins measurable}
		If either $ \kappa $ is measurable or $ \kappa=\omega $, then Maker has a winning strategy in the game $\mathfrak{MB}(K_{\kappa}, K_{\kappa}) $.
	\end{prop}
	\begin{proof}
		A sub-binary Hausdorff tree is a set theoretic tree $ T $ in which each vertex has at most two children and no two vertices at any limit level have the same set of predecessors.
		
		During the game Maker  builds  a sequence $ \left\langle T_\alpha: \alpha\leq \kappa  \right\rangle $ of sub-binary Hausdorff trees with root $ 0 $  and   $ T_\alpha\subseteq \kappa $ of height at most $1+\alpha$ such that 
		\begin{enumerate}[label=(\alph*)]
			\item\label{a condition}
			\begin{enumerate}[label=(\roman*)]
				\item $ T_0=\{ 0 \} $,
				\item  $ T_{\alpha+1} $ is obtained from $ T_\alpha $ by inserting a new maximal element,
				\item $ T_\alpha=\bigcup_{\beta<\alpha}T_\beta $ if $ \alpha $ is a limit ordinal,
			\end{enumerate} 
			\item\label{b condition} for every distinct $ <_{T_\alpha} $-comparable $ u,v\in T_\alpha $, the edge $ \{ u,v \}  $ is claimed by Maker in the game.
		\end{enumerate}
		
		Suppose that $ \alpha=\beta+1 $ and $ T_\beta $ is already defined.
		Maker picks the smallest  ordinal  $ v $ such that no edge incident with $v$ is claimed and claims  edge $ \{ 0, v \} $.
		Then, for as long as she can, on each following turn she connects $ v $ to vertices in  $ T_\beta $ in such a way that: 
		\begin{enumerate}
			\item\label{first cond} she maintains that the current neighbourhood of $ v $ in her graph is a downward closed chain in $ T_\beta $,
			\item\label{second cond} whenever she claims some $ \{ u,v \} $, then Breaker  has no edge between  $ v $ and the subtree $ T_{\beta, u} $ of $ T_\beta $ rooted at $ u $.
		\end{enumerate}
		Note that, at any step at which $v$ has a largest Maker-neighbour in $T_\beta$ and this neighbour has two children in $ T_\beta $, she can proceed.
		Moreover, she can also proceed even if there is no such largest Maker-neighbour as long as there is some element of $T_{\beta}$ whose predecessors are precisely the Maker-neighbours of $v$ in $T_{\beta}$.
		Thus, if Maker is unable to continue this process with $ v $, then either $ v $ has a largest Maker-neighbour in $ T_\beta $ which has at most one child or else there is no vertex in $T_{\beta}$ with precisely the Maker-neighbours of $u$ as its predecessors.
		In either case we can define $ T_{\beta+1} $  by adding $ v $ to $ T_\beta $ with its current set of Maker-neighbours as its predecessors, and Maker starts a new phase with a new fresh vertex.
		
		It is enough to show that there is a  $ \kappa $-branch $ B $ in $ T_\kappa $, because then $ G_M[B]  $ is a copy of $ K_{\kappa} $ by \ref{b condition}.
		Since $ \left|T_\kappa\right|=\kappa $ by \ref{a condition}, we can fix a $ \kappa $-complete free ultrafilter $ \mathcal{U} $ on $ T_\kappa $. 
		
		By transfinite recursion  we build a $ \kappa $-branch.
		Let $ v_0:=0 $.
		Suppose that there is an $ \alpha<\kappa $ such that the $ <_{T_\kappa} $-increasing  sequence $ \left\langle v_\beta: \beta<\alpha  \right\rangle $  is already defined and for each $ \beta<\alpha $,  $ T_{\kappa, v_\beta}\in \mathcal{U} $.
		If $ \alpha $ is a limit ordinal, then since $\bigcap_{\beta<\alpha}T_{v_\beta}\in \mathcal{U} $ by  the $ \kappa $-completeness of $ \mathcal{U} $, there is at least one vertex of $T$ with all $v_{\beta}$ as predecessors.
		We define $ v_\alpha $ to be the unique minimal such vertex, so that $ T_{v_\alpha}=\bigcap_{\beta<\alpha}T_{v_\beta}\in \mathcal{U} $.
		If  $ \alpha=\beta+1 $, then $ T_{\kappa, v_\beta}\in \mathcal{U} $ by assumption.
		Since $ T_\kappa $ is sub-binary,  $ v_\beta $ has a unique child $ v $ satisfying $ T_{\kappa, v}\in \mathcal{U} $ and we let $ v_{\beta+1}:=v $.
		The recursion is done and $\{ v_\alpha: \alpha<\kappa \} $ is clearly a $ \kappa $-branch.
	\end{proof}
	
	We remark that this strategy is quite flexible and deals also with a number of variants of the Maker-Breaker game.
	For example, if Breaker is allowed $k<\omega$ moves for every move that Maker picks, simply take a sub-$(k+1)$-regular Hausdorff tree, in which every node has at most $k+1$ children.
	Furthermore, if in addition Breaker is allowed to go first in every turn, simply weaken the Hausdorff assumption to the requirement that at most $k+1$ vertices at a limit level have the same set of predecessors.
	
	Since $ \omega_1 $ and $ \omega_2 $ are measurable cardinals under ZF+DC+AD (\cite{kanamori2008higher}*{Theorems 28.2 and 28.6}), the cases $ \kappa\in \{\omega, \omega_1 \} $ of Corollary~\ref{cor: Breaker win}/(\ref{i: Breaker win 2})  and the cases $\kappa \in \{\omega_1, \omega_2\}$ of Proposition \ref{p: Maker wins measurable} together imply Theorem \ref{t: main thm}.
	
	\subsection{Breaker may lose the \texorpdfstring{$\mathfrak{MB}(K_{\omega_1}, K_{\mathsf{club}})$}--game}
	\begin{prop}\label{prop: strategy steal}
		Under ZF+DC+AD, Breaker does not have a winning strategy in the game $ \mathfrak{MB}(K_{\omega_1}, K_{\mathsf{club}}) $.
	\end{prop}
	\begin{proof}
		First of all, the club filter  on $ \omega_1 $ is a countably complete free ultrafilter under ZF+DC+AD (this is explicit in the proof of \cite{kanamori2008higher}*{Theorem 28.2}).
		Furthermore, it is normal \cite{dorais2021bounding}*{Proposition 4.1}.
		Thus for any 2-colouring of $ [\omega_1]^{2} $ there exists a colour with a monochromatic $ K_{\mathsf{club}} $ (the standard proof of this for arbitrary normal ultrafilters uses only ZF, see \cite{jech2003set}*{Theorem 10.22}).
		It follows that if Breaker successfully prevents Maker from building a $ K_{\mathsf{club}} $, then he necessarily builds a $ K_{\mathsf{club}} $ himself. 
		
		Suppose for a contradiction that Breaker has a winning strategy.
		We shall show that Maker can ``steal'' this winning strategy. 
		Indeed, Maker picks an arbitrary edge in turn $ 0 $ as well as in each limit turn while in successor turns she pretends to be Breaker and claims edges according to his winning strategy.
		This is a winning strategy for Maker, a contradiction.
	\end{proof}
	Theorem \ref{t: main thm club} follows from the case $ \kappa=\omega $ of Corollary~\ref{cor: Breaker win}/(\ref{i: Breaker win 3}) and Proposition \ref{prop: strategy steal}.

	\begin{rem}
		The same strategy stealing argument shows that if $\kappa$ is a weakly compact cardinal, then Breaker does not have a winning strategy in the game $ \mathfrak{MB}(K_{\kappa}, K_{\kappa}) $.
	\end{rem}
	
	\begin{rem}
		We did not really use the full power of AD, just some consequences that are weaker in the sense of consistency strength than AD itself.
		The axiom-system ZF+DC+``$\omega_1$ is measurable'' is equiconsistent with ZFC+``there exists a measurable cardinal''  (see \cite{jech1968omega}).
		The club filter being an ultrafilter is a strictly stronger assumption, for more details see  p. 3 in \cite{dorais2021bounding}.
	\end{rem}

	\bibliographystyle{unsrtnat}
	\bibliography{ref}
\end{document}